\documentclass{amsart}

\usepackage{pinlabel} 
\usepackage{amsmath}
\usepackage{amsfonts}
\usepackage{amssymb}
\usepackage{mathtools}
\usepackage{amscd}
\usepackage{amsthm}
\usepackage[all,cmtip]{xy}
\usepackage{anysize}
\usepackage{epsfig}
\usepackage[colorinlistoftodos]{todonotes}
\usepackage{xcolor}
\usepackage[utf8]{inputenc}

\usepackage[normalem]{ulem}

\newtheorem{theorem}{Theorem}[section]
\newtheorem{proposition}[theorem]{Proposition}
\newtheorem{corollary}[theorem]{Corollary}
\newtheorem{lemma}[theorem]{Lemma}

\theoremstyle{definition}
\newtheorem{problem}[theorem]{Problem}

\newtheorem{example}[theorem]{Example}
\newtheorem{remark}[theorem]{Remark}

\newcommand{\dbC}{\mathbb{C}}

\newcommand{\dbZ}{\mathbb{Z}}

\newcommand{\calF}{{\mathcal F}}

\newcommand{\calG}{{\mathcal G}}

\newcommand{\orf}[1]{\Or(#1,\calF)}
\newcommand{\orgf}{\Or(G,\calF)}
\newcommand{\org}{\Or(G)}

\newcommand{\triv}{\textsc{Tr}}
\newcommand{\fin}{\textsc{Fin}}
\newcommand{\vcyc}{\textsc{Vcyc}}
\newcommand{\all}{\textsc{All}}

\newcommand{\func}[3]{#1\colon #2\to#3}
\newcommand{\inv}[1]{ #1^{-1}}

\newcommand{\nbeq}{\begin{equation}}
\newcommand{\neeq}{\end{equation}}
\newcommand{\beq}{\begin{equation*}}
\newcommand{\eeq}{\end{equation*}}

\newcommand{\adamson}[4]{H_{#1}([#2:#3];#4)}
\newcommand{\takasu}[4]{H_{#1}(#2,#3;#4)}

\newcommand{\CC}[2][*]{\mathbf{#2}_{#1}}

\newcommand{\homo}[1][\bullet]{\ensuremath{H_{#1}}}

\newcommand{\Hom}[1][]{\mathrm{Hom}_{#1}}
\newcommand{\Tor}[2][]{\mathrm{Tor}^{#1}_{#2}}
\newcommand{\Ext}[2][]{\mathrm{Ext}_{#1}^{#2}}

\DeclareMathOperator{\id}{\mathrm{Id}}
\DeclareMathOperator{\gmap}{map_\mathit{G}}
\DeclareMathOperator{\lmodu}{\mathbf{--mod}}
\DeclareMathOperator{\rmodu}{\mathbf{mod--}}

\DeclareMathOperator{\abe}{\mathit{Ab}}
\DeclareMathOperator{\spaces}{\mathit{Top}}
\DeclareMathOperator{\chain}{\mathbf{Ch}}

\DeclareMathOperator{\res}{res}
\DeclareMathOperator{\ind}{ind}
\DeclareMathOperator{\coind}{coind}
\DeclareMathOperator{\Or}{Or}

\newcommand{\defeq}{\vcentcolon=}

\renewcommand{\emph}[1]{\textbf{#1}}

\begin{document}

\makeatletter
\providecommand\@dotsep{5}
\def\listtodoname{List of Todos}
\def\listoftodos{\@starttoc{tdo}\listtodoname}
\makeatother

\title[Relative group homology theories with coefficients]{Relative group homology theories with coefficients and the comparison homomorphism}

\author{Jos\'e Antonio Ariciniega-Nev\'arez}
\address{Divisi\'on de Ingenier\'ias, Universidad de Guanajuato}
\email{ja.arciniega@ugto.mx}

\author{Jos\'e Luis Cisneros-Molina}
\address{Unidad Cuernavaca del Instituto de Matem\'aticas, National University of Mexico, Mexico 62210}

\email{jlcisneros@im.unam.mx}
\thanks{The second author is Regular Associate of the Abdus Salam International Centre for Theoretical Physics, Trieste, Italy. Supported by project CONACYT~253506.}

\author{Luis Jorge S\'anchez Salda\~na }
\address{Unidad Cuernavaca del Instituto de Matem\'aticas, National University of Mexico, Mexico 62210}

\email{luisjorge@im.unam.mx}

\thanks{The third named author was supported by  DGAPA-UNAM postdoctoral grant.}


\date{}


\begin{abstract}
Let $G$ be a group, let $H$ be a subgroup of $G$ and let $\Or(G)$ be the orbit category. In this paper we extend the definition of the relative group homology theories of the pair $(G,H)$ defined by Adamson and Takasu 
to have coefficients in an $\Or(G)$-module. There is a canonical comparison homomorphism defined by Cisneros-Molina and Arciniega-Nevárez from Takasu's theory to Adamson's theory. We give a necessary and sufficient condition 
on the subgroup $H$ for which the comparison homomorphism is an isomorphism for \textit{all} coefficients. We also use the L\"uck-Wiermann construction to introduce a long exact sequence for Adamson homology. Finally, 
we provide some examples of explicit computations for the comparison homomorphism.
\end{abstract}

\maketitle

\section{Introduction}\label{sec:intro}
Let $G$ be a group, let $H$ be a subgroup of $G$, and let $M$ be an arbitrary $G$--module. In the literature there are two \textit{relative group (co)homology theories}   associated to the pair $(G,H)$ with coefficients in $M$. Both reduce to the classical group (co)homology of $G$ when $H$ is the trivial subgroup. The first relative group (co)homology theory, denoted by $H_*([G:H];M)$, was defined by Adamson \cite{Adamson:CTNNSNNF} and later Hochschild \cite{Hochschild:RHA} interpreted 
Adamson's theory in terms of relative homological algebra. The second relative group (co)homology theory, denoted by $H_*(G,H;M)$, was introduced by Massey \cite[Problem 22]{massey:SPATTFB} and later studied by 
Takasu in \cite{Takasu-OTCOFITHA} and \cite{takasu-RHARCTOG}. The reader can look in \cite{AC17} for more details about the history and references. From now on, we call the former \textit{Adamson relative 
group (co)homology theory} and the latter \textit{Takasu relative group (co)homology theory}. 

To the best of our knowledge Adamson relative homology theory and Takasu relative homology theory were compared for the first time in  \cite{AC17}.  In fact, in \cite[Section~7]{AC17} Arciniega-Nevarez and Cisneros-Molina defined a canonical homomorphism from Takasu homology theory to Adamson homology theory $\varphi:H_*(G,H;\dbZ)\to H_*([G:H];\dbZ)$, which from now on we call \textit{the comparison homomorphism}. They gave a  sufficient condition for $\varphi$ to be an isomorphism: if $H$ is a malnormal subgroup of $G$ then, for all $i\geq 1$, the comparison homomorphism 
$\varphi\colon\adamson{i}{G}{H}{\dbZ}\to \takasu{i}{G}{H}{\dbZ}$ is an isomorphism. 

Recall that classical group (co)homology can be defined topologically as the (co)homology groups of the classifying space $BG$ of $G$ with local coefficients associated to $M$. In analogy with this, 
taking $M$ as a \textit{trivial} $G$-module,  topological definitions of Adamson and Takasu relative (co)homology theories were given \cite{Blowers:CSPR,takasu-RHARCTOG,AC17}.
The topological definition of Takasu relative group (co)homology extends without problem to consider coefficients in an \textit{arbitrary} $G$-module $M$ taking local coefficients associated to $M$, but the topological definition of 
Adamson relative group (co)homology only works with \textit{trivial} coefficients. So, a natural question is if it is possible to give a topological definition for Adamson relative group (co)homology 
with ``more general coefficients'', such that, when taking trivial coefficients it coincides with the topological definition given in \cite{AC17}, and when taking coefficients in an arbitrary $G$-module, it 
coincides with Adamson's original algebraic definition.

The solution to the questions above mentioned is to define both Adamson and Takasu relative homology theories with coefficients in an $\Or(G)$--module, 
using Bredon (co)homology, where $\Or(G)$ is the orbit category. Also, once we consider more general coeffcients, we prove the converse of  Arciniega-Nevarez and Cisneros-Molina's theorem (see Theorem \ref{comparisontheorem}): the comparison homomorphism is an isomorphism for \textit{all} $\Or(G)$-module coefficients if and only if $H$ is a malnormal subgroup of $G$. It is worth saying that the comparison of both relative (co)homology theories is not only motivated by curiosity. Recently, Adamson and Takasu theories have been used to construct invariants of complete hyperbolic $3$-manifolds of finite volume \cite{Zickert:VCSIR,Arciniega-Cisneros:IH3MRGH} and in relation with quandle homology \cite{Inoue-Kabaya:QHCV,Nosaka:QTP}. For instance, in \cite[Lemma~7]{Inoue-Kabaya:QHCV} we can see a natural application of the comparison homomorphism.

The main feature of Adamson relative (co)homology is that, provided $H$ is a normal subgroup of $G$, it is isomorphic to the classical group (co)homology of $G/H$ (see \cite[Theorem~3.2] {Adamson:CTNNSNNF}, 
\cite[Section~6]{Hochschild:RHA} and \cite[Corollary~4.29]{AC17}). For Takasu relative (co)homology this is no true in general. The main property of Takasu relative (co)homology is 
a long exact sequence, which (via the topological definition) corresponds to that of a pair of topological spaces in singular homology (\cite[Proposition~2.3]{takasu-RHARCTOG}). Such a sequence is not available in 
Adamson homology. We show that the extension of Adamson and Takasu relative homologies presented here still satisfy their corresponding property.

As in the classical group homology theory, there are descriptions of $\adamson{i}{G}{H}{M}$ and $\takasu{i}{G}{H}{M}$ using derived functors once we are able to define exact 
sequences and projective resolutions in a suitable context. This was done in Takasu's and Hochshild's original papers. In \cite[Section 7.1]{AC17}, there is also a description of 
the comparison homomorphism in a purely algebraic setting using derived functors. In the present paper we also describe Adamson and Takasu homology with coefficients in an 
$\Or(G)$-module using the language of derived functors.

The paper is organized as follows. In Section~\ref{sec:T.A} we recall the definitions of Takasu and Adamson relative group homology theories, and we show why the topological definition for Adamson's theory
given in \cite{AC17} does not work for coefficients in an arbitrary $G$--module. In Section~\ref{sec:ooc} we recall the definition of the (restricted) orbit category $\Or(G,\calF)$ as well as the definition of $\Or(G,\calF)$--modules and 
$\Or(G,\calF)$--spaces. Section~\ref{sec:G-ht} is devoted to define Adamson and Takasu homology for a 
pair $(G,H)$ with coefficients in a $\Or(G)$--module $M$, as far as we know these homology theories have not been compared in this context; later on, we define the comparison 
homomorphism and prove in Theorem \ref{comparisontheorem}, a converse of Cisneros-Molina and Arciniega-Nev\'arez theorem \cite[Theorem~7.13]{AC17} for the comparison homomorphism. In Section~\ref{sec:derived-functors} we provide descriptions 
of Adamson and Takasu homology using derived functors in suitable categories, as well as an algebraic description of the comparison homomorphism.  Section~\ref{LWconstruction} has as a main goal to describe a long exact sequence for Adamson homology, analogue to that of Takasu homology, 
using the L\"uck--Weiermann construction from \cite{LW12}. Finally, in Section~\ref{sec:exa} we provide explicit computations of the comparison homomorphism; in particular, 
we show that the comparison homomorphism associated to $(C_4,C_2)$ (with constant coefficients) is not an isomorphism, therefore, at least to our knowledge,  in this concrete case, the isomorphism 
between Adamson and Takasu homology seems to be a coincidence.

Every result  and every construction in this paper has its cohomological version. 

\section{Adamson and Takasu relative group homology theories}\label{sec:T.A}
In this article, we always consider \emph{groups} as topological spaces endowed with the \emph{discrete topology}. Throughout this section $G$ will denote a group and $H$ a subgroup of $G$.

In this section we recall the definitions of Takasu and Adamson relative group homology theories for the pair $(G,H)$. 
The topological definition for Takasu relative group homology
given in \cite[\S 5.1]{AC17} for \textit{trivial} coefficients, also works with coefficients in an arbitrary $G$--module $M$ taking homology with local coefficients associated to $M$.
We also recall the topological definition of Adamson relative group homology given in \cite[\S 4.3]{AC17} for \textit{trivial} coefficients, and we see why it \textit{cannot} be extended
to use coefficients in an arbitrary $G$-module $M$ taking homology with local coefficients associated to $M$.

\subsection{$G$-spaces}

Let $X$ and $X'$ be  two $G$--spaces. We denote by $\gmap(X,X')$ the set of $G$--maps from $X$ to $X'$. Let $\pi\colon K\to G$ be a homomorphism of groups, we denote $\res_\pi X$ the space $X$ with the natural $K$-action induced by $\pi$. If $\pi$ is the inclusion we denote the respective space by $\res_K^G$. If $Y$ is an $H$--space, the \emph{induction} $\ind_H^G Y$ is the $G$-space $G\times Y$  divided by the $H$--action $(g,x)\cdot h=(gh,h^{-1} x)$.

\subsection{Classifying spaces for families}

Let $G$ be a group, a \emph{family of subgroups of $G$} is a nonempty collection $\mathcal{F}$ of subgroups of $G$ which is closed under conjugation and taking subgroups. Examples of families are $\triv$ the family containing only the trivial subgroup, and $\fin$, 
$\vcyc$, $\all$ the families of finite, virtually-cyclic and all subgroups respectively.

In the present work, we are interested in the family $\calF(H) =\lbrace K\leq G | g^{-1}Kg\leq H \text{ for some } g\in G \rbrace$ of subgroups of $G$ that we call the family generated by $H$. Given a family $\mathcal{F}$ of subgroups of $G$, define $\calF\cap H = \lbrace L\cap H ~|~ L\in \calF \rbrace.$

Given a group $G$ and a family of subgroups $\calF$, a model for the \emph{classifying space} $E_\mathcal{F} G$ is a $G$-CW--complex $X$ satisfying:
\begin{itemize}
\item The isotropy group $G_x$ belongs to $\mathcal{F}$, for all $x\in X$, and
\item the fixed point set $X^H$ is contractible for every $H\in\calF$.
\end{itemize} 
 
Equivalently, a model for $E_\mathcal{F} G$ is a terminal object in the $G$--homotopy category of $G$-CW-complexes with isotropy in $\calF$, sometimes called $G$-$\calF$-CW--complexes. 
Given a group $G$ and a family of subgroups $\mathcal{F}$, 
there always exists a model for $E_\mathcal{F} G$ and is unique up to $G$--homotopy equivalence \cite[Theorem 1.9]{Lu05}.

\begin{remark}\label{rem:EG}
When $\calF=\triv$, we have that $E_{\calF}G$ corresponds to the \textit{universal bundle} $EG$ of $G$.
The $G$--orbit space of $EG$ is the classical classifying space $BG$ of $G$. 
In analogy with  $BG$, we denote by $B_{\calF}G$ the $G$--orbit space of $E_{\calF}G$.
\end{remark}

\subsection{Takasu relative group homology}

Denote by $G\lmodu$ (resp.~$\rmodu G$) the category 
of left (resp.~right) $G$--modules, and given two $G$--modules $M$ and $M'$ we denote by $\gmap(M,M')$ the set of $G$--module homomorphisms from $M$ to $M'$.

In \cite[Ch.~I \S1]{takasu-RHARCTOG} Takasu defines the $G$--module $I_{(G,H)}(\dbZ)$ to be the kernel of the augmentation homomorphism $\dbZ[G/H]\to \dbZ$. 

Let $M$ be an \textit{arbitrary} $G$--module. The \emph{Takasu relative homology of the pair $(G,H)$ with coefficients in $M$} is given by
\begin{equation}\label{eq:Takasu.Tor}
\begin{aligned}
\homo[n](G,H;M)&= \Tor[G]{n-1}(I_{(G,H)}(\dbZ),M).
\end{aligned}
\end{equation}

The classifying space $BH$ can be regarded as a
subspace of the classifying space $BG$. In fact, there is a map $\func{\iota}{BH}{BG}$ induced by the inclusion of $H$ in $G$; the \emph{mapping cylinder} 
$\mathrm{Cyl}(\iota)$ of $\iota$ is a model for $BG$ since it is homotopically equivalent to $BG$ and it clearly contains $BH$ as subspace. 
Then, an alternative definition for the \emph{Takasu relative homology with coefficients in $M$} is defined as the homology of the pair of spaces $(BG,BH)$ with local coefficients in $M$. That is
\begin{equation}\label{eq:Takasu.top.def}
\begin{aligned}
 \homo[n](G,H;M)&=\homo[n](BG,BH;M).
\end{aligned}
\end{equation}
It is clear that when $H$ is the identity subgroup we recover the (reduced) homology of the group $G$ with coefficients in $M$.

\begin{problem}\label{prob:cone}
When $M$ is a \textit{trivial} $G$-module one can consider the \emph{mapping cone} $\mathrm{Cone}(\iota)=\mathrm{Cyl}(\iota)/BH$ of $\iota$, so taking reduced homology we have that
\begin{align*}
 \homo[n](G,H;M)&=\homo[n](\mathrm{Cyl}(\iota),BH;M)=\tilde{H}_n(\mathrm{Cone}(\iota);M).
\end{align*}

But when $M$ is an \textit{arbitrary} $G$--module we cannot use $\mathrm{Cone}(\iota)$. Using Seifert--van Kampen Theorem to compute the fundamental group of 
$\mathrm{Cone}(\iota)$ we have that $\pi_1(\mathrm{Cone}(\iota))=G/N$ with $N$ the normal subgroup generated by $H$. So in order to compute $H_n(\mathrm{Cone}(\iota);M)$ the homology of $\mathrm{Cone}(\iota)$ with
local coefficients associated to $M$,  the module $M$ has to be a $G/N$--module, but in defining $\homo[n](BG,BH;M)$ we can use any $G$--module! 
\end{problem}

\subsection{Adamson relative group homology}

Consider the following chain complex $(\CC{C}(G/H),\partial_*)$ of $G$--modules: let $\CC[n]{C}(G/H)$ be 
the free abelian group generated by the ordered $(n+1)$--tuples of elements of $G/H$; define the $i$--th face homomorphism $d_i\colon\CC[n]{C}(G/H)\to\CC[n-1]{C}(G/H)$ by
$d_i(g_0H,\dots,g_nH)=(g_0H,\dots,\widehat{g_iH},\dots,g_nH)$, where $\widehat{g_iH}$ denotes deletion, and the boundary homomorphism $\partial_n\colon\CC[n]{C}(G/H)\to\CC[n-1]{C}(G/H)$ by
$\partial_n=\sum_{i=0}^n (-1)^id_i$. 
We have that the augmented complex
\begin{equation}\label{eq:crpr}
\xymatrix{\cdots\ar[r]& \CC[n]{C}(G/H)\ar[r]&\cdots\ar[r]& \CC[2]{C}(G/H)
\ar[r]^{\partial_2}&\CC[1]{C}(G/H)\ar[r]^{\partial_1}&\CC[0]{C}
(G/H)\ar[r]^{\varepsilon}&\dbZ\ar[r]&0,}
\end{equation}
is acyclic \cite[Proposition~3.2]{AC17}. Hence given an \textit{arbitrary} $G$--module $M$ define $$ \CC{B}(G/H;M)=\CC{C}(G/H)\otimes_{\dbZ[G]}M. $$

The \emph{Adamson relative homology with coefficients in $M$} is given by
\begin{equation}\label{eq:Adamson.alg.def}
\begin{aligned}
 \homo[n]([G:H];M)&=\homo[n](\CC{B}(G/H;M)).
\end{aligned}
\end{equation}
This is the definition given by Adamson in \cite[\S3]{Adamson:CTNNSNNF}. It is clear that when $H$ is the identity subgroup the complex $\CC{C}(G/H)$ 
is the canonical free $G$--resolution of $\dbZ$ and we recover the classical group homology. Instead of $G/H$, we can use any other isomorphic $G$--set $X_{(H)}$, in this case, 
we call $\CC{C}(X_{(H)})$ \emph{the standard complex of $(G,H)$}, then we can see Adamson relative group homology as a particular case of the homology of a permutation 
representation defined by Snapper \cite{Snapper:CPRISS}. Hochschild \cite{Hochschild:RHA} interpreted 
Adamson's theory in terms of relative homological algebra by proving that the complex \eqref{eq:crpr} is a relative projective resolution of $\dbZ$ \cite[Proposition~4.11]{AC17}.

There is also a topological definition, given in \cite{AC17}, when $M$ is a \textit{trivial} $G$--module. Consider the family
of subgroups $\mathcal{F}(H)$ generated by $H$. Then
\begin{align*}
\homo[n]([G:H];M)&= \homo[n](B_{\calF (H)}G;M).
\end{align*}

\begin{problem}\label{problemtopadamson}
To take coefficients in an \textit{arbitray} $G$--module $M$ as in the algebraic definition, we cannot simply take homology with local coefficients associated to the module $M$. 
There is a problem analogous to Problem~\ref{prob:cone} in the case of Takasu's theory; to take homology with local coefficients associated to a module $M$, the module has to be a 
$\pi_1(B_{\calF (H)}G)$--module. Since $\pi_1(B_{\calF (H)}G)=G/N$  where $N$ is the normal subgroup of $G$ generated by $H$ \cite[Proposition~4.23]{AC17},
given a $G$--module $M$ we need to modify it in order to use it as coefficient. But in the algebraic definition we can take any $G$--module $M$! 
\end{problem}

In the present article, we give a suitable topological definition using Bredon homology with coefficients in any module over the orbit category to solve (in a more general setting) 
Problems~\ref{prob:cone} and \ref{problemtopadamson}.

In \cite[Section~7]{AC17} the authors defined the so-called comparison homomorphism $\varphi:H_*(G,H;\dbZ)\to H_*([G:H];\dbZ)$. We prove that there is a comparison homomorphism in our setting. We have the following natural questions that we address in the present work.

\begin{problem}\label{problem:comparison homomorphism}
Under what conditions $\varphi:H_*(G,H;\dbZ)\to H_*([G:H];\dbZ)$ is an isomorphism? If the comparison homomorphism is not an isomorphism, is there a good description for the kernel and the cokernel?
\end{problem}

\section{Objects over the orbit category}\label{sec:ooc}

In this section we introduce the (restricted) orbit category, as well as some objects over the orbit category. All the material of this section can be found in great detail  in \cite{MV03}, \cite{tomDieck:TransGrp} and \cite{Lu89}. Throughout this section $G$ will denote a group and $\calF$ a family of subgroups of $G$. 

The \emph{restricted orbit category} $\orf{G}$ is the category whose objects are homogeneous spaces (also called orbits) $G/H$ with 
$H\in \calF$, and whose morphisms are $G$--maps. The set of $G$--maps between the orbits $G/H$ and $G/K$ is denoted by $\gmap(G/H, G/K)$. We denote $\Or(G,\all)$ simply by $\org$. Note that, for every family $\calF$, we have a canonical inclusion $\Or(G,\calF)\hookrightarrow \Or(G)$.

It is easy to see that every element in $\gmap(G/H,G/K)$ is of the form $\func{R_a}{G/H}{G/K}$,
$gH \mapsto ga^{-1}K$, provided $aHa^{-1} \subseteq K$. Also $R_a=R_b$ if and only if $ab^{-1} \in K$, and $R_b \circ R_a= R_{ba}$, whenever the composition makes sense.

\subsection{Modules over the orbit category}

A covariant (resp. contravariant) \emph{$\orf{G}$--module} is a covariant (resp. contravariant) functor from  $\orf{G}$ to the category of abelian groups. 
A \emph{morphism} $M \to N$ of $\orf{G}$--modules of the same variance is a natural transformation between the underlying functors. We denote by $\Hom[\orf{G}](M,N)$ the set of all morphisms $M \to N$. We denote by $\orgf\lmodu$ (resp. $\rmodu\orgf$) the category of covariant (resp. contravariant) $\orgf$--modules. These are abelian categories \cite[Proposition~IX~3.1]{MacLane:Homotopy} with enough projectives \cite[p.~10]{MV03}. 
  
\begin{example}\label{z-trasformation}
Denote the free abelian group with basis $\gmap(G/H, G/K)$, by $\dbZ[G/H, G/K]$.
So, $\dbZ[G/H,-]$ and $\dbZ[-, G/K]$ respectively define a covariant and a contravariant $\orf{G}$--\-mo\-du\-les. The module $\dbZ[-, G/K]$ happens to be a \emph{free} contravariant $\orgf$--module.
\end{example}

\begin{remark}\label{naturality-modules}
The categories $\rmodu \Or(G,\triv)$ (resp. $\Or(G,\triv)\lmodu$) and
 $\rmodu G$ (resp. $G\lmodu$) are canonically isomorphic. To the $\Or(G,\triv)$--module $M$ corresponds the $G$--module $M(G/I)$, where $I$ denotes the trivial subgroup of $G$.
\end{remark}

Given a contravariant $\orf{G}$--module $M$ and a covariant $\orf{G}$--module $N$, there is a \emph{tensor product} $M\otimes_{\orf{G}} N$, which is a suitable abelian group (see \cite[p.~14]{MV03}). Hence we have a tensor product functor $-\otimes_{\orf{G}} N$ from the category of contravariant $\orf{G}$--modules to the category of abelian groups. For every $\orgf$--module $M$ with a suitable variance and for all $H\in \calF$, we have a \emph{Yoneda-type isomorphism} (see \cite[p.~9, p.~14]{MV03})
\begin{equation}\label{Yoneda}
\begin{aligned}
\dbZ[-,G/H]\otimes_{\orgf} M&=M(G/H).
\end{aligned}
\end{equation}

\subsection{Restriction, induction and coinduction for modules over the orbit category}
Let  $\varphi: H \to G$ be a homomorphism of groups. Denote $\varphi^* \calF$ the family of subgroups of $H$ that are mapped by $\varphi$ to a group in $\calF$. We have a natural functor
\begin{equation*}
\overline{\varphi}:\Or(H,\varphi^*\calF)\to \Or(G, \calF))
\end{equation*}

Given an $\Or(G,\calF)$--module $M$,  the \emph{restriction functor} $\res_\varphi(M)$ is the $\Or(H,\calF\cap H)$--module of the same variance as $M$, defined by
$\res_\varphi(M)= M\circ \bar \varphi$. It is no difficult to see that $\res_\varphi$ defines covariant functors
\begin{gather*}
\res_\varphi\colon \rmodu\Or(G,\calF)\to \rmodu\Or(H,\varphi^*\calF),\\
\res_\varphi\colon \Or(G,\calF)\lmodu\to\Or(H,\varphi^*\calF)\lmodu.
\end{gather*}

We also have induction and coinduction functors (see  \cite{Lu89} for the definitions)
\begin{gather*}
\ind_\varphi\colon\Or(H,\varphi^*\calF)\lmodu \to\Or(G,\calF)\lmodu,\\
\coind_\varphi\colon\rmodu\Or(H,\varphi^*\calF) \to\rmodu\Or(G,\calF).
\end{gather*}
such that all the usual adjoint properties hold. If $\varphi$ is an inclusion, then we use the notation $\res_H^G$, $\ind_H^G$, and $\coind_H^G$.

\subsection{Coinvariants}
Given a $G$--module $M$, we define the \emph{coinvariants functor} $\underline{M}=\dbZ[-]\otimes_{\dbZ G}M$. We have that $\underline{M}(G/K)$ are the $K$-coinvariants $M_K$ of $M$.
This defines a covariant functor from the category of left $G$--modules to the category of covariant $\Or(G,\calF)$--modules.  

\subsection{Spaces over the orbit category}\label{sssec:res.ind.sp}
A covariant (resp. contravariant) \emph{$\orf{G}$--space} is a covariant (resp. contravariant) functor $\orf{G}\to \spaces$, where $\spaces$ is the category of topological spaces.

If $X$ is a $G$--space we can define the \emph{fixed point} contravariant $\Or(G,\calF)$--space $\overline{X}=\gmap(-,X)$. We have that $\overline{X}(G/K)=X^K$.
This defines a functor from the category of $G$--spaces to the category of $\Or(G,\calF)$-spaces.

\section{Adamson and Takasu theories using Bredon homology}\label{sec:G-ht}

In this section we introduce Bredon homology for $G$-CW--complexes. The material of this section can be found in great detail in \cite{MV03}, \cite{SG05}. 
In this section we define  the relative group homology theories of Adamson and Takasu with coefficients in an $\Or(G)$--module, as well as the comparison homomorphism between them. 
Such definitions generalize those existing in the literature.

\subsection{Bredon homology}
If $X$ is a $G$-CW--complex then we have associated the fixed point functor, i.e. a contravariant $\org$--space $\overline{X}\colon \Or(G)\to \spaces$, 
and the contravariant functor $\CC{C}(X)(-):\Or(G) \to \chain(\abe)$, which is the composition of $\overline{X}$  with the cellular chain complex functor (Subsection~\ref{sssec:res.ind.sp}). 
Let $M$ be a covariant $\org$-module, we can define a chain complex of abelian groups by considering the tensor product 
\begin{equation*}
\CC{C}(X;M)\defeq \CC{C}(X)\otimes_{\org} M.
\end{equation*}

We describe $\CC[i]{C}(X;M)$ in more detail, for further details about the boundary morphisms see \cite{Bredon:ECT}. Let $\Delta_i$ be the set of $i$--cells of $X$, since $G$ acts cellularly on $X$, 
the set $\Delta_i$ is a $G$--set. Let $K_{\sigma}$ be the isotropy group of $\sigma \in \Delta_i$. Let $\Sigma_i$ be a set of representatives for the $G$--orbits in $\Delta_i$. Then, using the distributive property of the tensor product and the Yoneda-type isomorphism  
\begin{align*}
\CC[i]{C}(X) \otimes_{\org} M&=\bigoplus_{\sigma\in \Sigma_i}(\dbZ[-,G/K_{\sigma}]\otimes_{\org} M)\\ &=\bigoplus_{\sigma\in \Sigma_i}M(G/K_{\sigma}).
\end{align*}

By definition, the \emph{Bredon homology} of the $G$-CW--complex $X$ with coefficients in the $\org$--module $M$ (with suitable variance) is
\begin{gather*}
 H_*^G(X;M) = H_*(\CC{C}(X)\otimes_{\org} M).
\end{gather*}
In a complete analogous way, for a $G$-CW-pair $(X,Y)$.

\begin{example}\label{gmodulecoefficients}
If $M$ is a $G$--module, we have the coinvariants functor $\underline{M}$. Hence, we can define the \emph{homology of 
the $G$-CW--complex $X$ with coefficients in the $G$--module $M$} by 
\begin{gather*}
H_*^G(X;M)=H_*^G(X;\underline{M}).
\end{gather*}
\end{example}

\begin{example}\label{bredon:coef:gmodule}
Let $X$ be a $G$-CW--complex and $M$ a $G$--module. Then the chain complex of abelian groups $\CC{C}(X;\underline{M})\defeq \CC{C}(X)\otimes_{\org} \underline{M}$, is 
isomorphic to the tensor product $ \CC{S}(X)\otimes_{\dbZ[G]} M $, where $\CC{S}(X)$ is the classical cellular chain complex of $X$ with the induced action of $G$. In fact, 
this can be easily seen by decomposing $\mathbf{S}_{i}(X)$ as a direct sum of $G$--modules of the form $\dbZ[G/K]$, the Yoneda isomorphism (\ref{Yoneda}), and the isomorphism $\dbZ[G/K]\otimes_{\dbZ[G]}M=M_K$:
\begin{align*}
\CC{S}(X)\otimes_{\dbZ[G]} M &= \bigoplus_{\sigma\in \Sigma_*}(\dbZ[G/K_{\sigma}]\otimes_{\dbZ[G]} M)\\ &=\bigoplus_{\sigma\in \Sigma_*}M_{G/K_{\sigma}}\\
&=\bigoplus_{\sigma\in \Sigma_*}(\dbZ[-,G/K_\sigma]\otimes_{\org}\underline{M}).
\end{align*}
where $\Sigma_*$ is, as before, a set of representatives for the $G$–orbits in $\Delta_i$. The computation of the boundary homorphism is straightforward.
\end{example}

\begin{example}\label{bredongmodules}
If $A$ is a trivial $G$--module (i.e. $ga=a$ for all $g\in G$ and $a\in A$) the functor of coinvariants $\underline{A}$ asociated to $A$ is the constant $\Or(G)$--module given by $\underline{A}(G/K)=A$ for 
all subgroups $K$, and $\underline{A}(R_a)=Id$ for every morphism in $\org$. By Example~\ref{bredon:coef:gmodule} Bredon homology of $X$ with coefficients in $\underline{A}$ recovers 
the cellular homology of $X/G$ with coefficients in the trivial $G$--module $A$.
\begin{gather*}
H^G_*(X;A)=H_*^G(X;\underline{A})=H_*(X/G;A).
\end{gather*}
\end{example}

\begin{example}\label{localcoefficienthomology}
Let $(X,Y)$ be a pair of $G$-CW--complexes with $X$ simply-connected and with free $G$--action, and let $M$ be a covariant $\org$--module. Then $H_*^G(X,Y;M)$ is isomorphic 
to the homology of the pair $(X/G,Y/G)$ with local coefficients associated to the $G$--module $M(G/I)$. In fact, since $X$ is simply-connected we know that 
the canonical projection $X\to X/G$ is the universal covering projection so that $G\cong \pi_1(X/G)$.
\end{example}

The following theorem will be useful in the next section in order to establish the main properties of Adamson and Takasu relative homology theories.

\begin{lemma}\label{inductivestructureforbredon}
Let $G$ be a group and $H$ a subgroup of $G$. Consider a covariant $\Or(G)$--module $M$. Then, for any $H$-CW--complex $X$, there are natural isomorphisms
\begin{gather*}
H_*^H(X;\res_H^G(M))\cong H_*^G(\ind_H^G X;M).
\end{gather*}
\end{lemma}
\begin{proof}
The proof follows by applying the definition of Bredon homology and the classical adjunction properties of $\res_H^G$ and $\ind_H^G$ (see \cite[Lemma 1.9]{DL98}).
\end{proof}

\subsection{Adamson relative group homology}\label{ssection:Adamsonhomology}
Consider a discrete group $G$, a subgroup $H$ of $G$, and an $\org$--module $M$ (with suitable variance). Recall that $\calF(H)$ is the family of subgroups of $G$ generated by $H$. 
Then we define the \emph{Adamson relative group homology} of the pair $(G,H)$ with coefficients in $M$ by
\begin{gather*}
H_*([G:H];M)\defeq H_*^G(E_{\calF (H)} G;M).
\end{gather*}

\begin{remark}
Note that:
\begin{itemize}
\item By Example~\ref{bredongmodules}, our definition reduces to the definition of Adamson relative group homology with coefficients in a trivial $G$--module given in \cite{AC17}. 
\item For $H$ trivial, we have $\calF(H)=\triv$ and, by Remark~\ref{rem:EG}, we get the universal covering $EG$ of the classical classifying space $BG$ of $G$, we recover the classical homology of $G$ with coefficients in the $G$--module $M(G/I)$.
\item If $M$ is a $G$--module, then we recover the definition of Adamson homology with coefficients in a $G$--module given in \eqref{eq:Adamson.alg.def}  
(see Example \ref{gmodulecoefficients} and Example \ref{bredon:coef:gmodule}).
\end{itemize}
\end{remark}

Now we state the main property of Adamson relative group homology. Roughly speaking, it des\-cribes an excision phenomenon. As a consequence the Adamson homology of $(G,H)$  is the group 
homology of the quotient $G/H$ if $H$ is a normal subgroup of $G$. 

\begin{proposition}\label{prop:3.IT}
Let $N$ be a normal subgroup of $G$ contained in $H$, and let $\pi\colon G\to G/N$ be the quotient projection. Let $M$ be an $\Or(G)$--module. Then, for all $n\geq 0$, we have the following isomorphisms
\begin{align*}
\homo[n]([G:H];M)&\cong \homo[n]([G/N:H/N];\ind_\pi (M)).
\end{align*}
\end{proposition}

\begin{proof}
Consider $X$ a model for $E_{\calF(H/N)}G/N$, then $\res_\pi X$ is a model for $E_{\calF(H)}G$.
It is not difficult to see that $C_*(\res_\pi X)=\res_\pi C_*(X)$ as chain complexes in $\rmodu \Or(G)$. Now, using \cite[Lemma 3.1  and Proposition 3.2]{MP02}, we have the following isomorphisms
\begin{align*}
\res_\pi C_*(X) \otimes_{\Or(G)} M  &\cong ( C_*(X) \otimes_{\Or(G/N)} \dbZ [\pi(-),-]  )\otimes_{\Or(G)} M \\
  &\cong  C_*(X) \otimes_{\Or(G/N)} (\dbZ [\pi(-),-]  \otimes_{\Or(G)} M)\\
  & \cong C_*(X) \otimes_{\Or(G/N)}\ind_\pi M.
\end{align*}
Applying homology to both sides we get the conclusion for Adamson homology. 
\end{proof}

Using Example \ref{bredongmodules}, and a straightforward computation of the coefficients, we get the following corollary.

\begin{corollary}\label{adamsongmodules}
Let $N$ be a normal subgroup of $G$ contained in $H$. Let $M$ be a $G$--module. Then, for all $n\geq 0$, we have the following isomorphisms
\begin{align*}
\homo[n]([G:H];M)&\cong \homo[n]([G/N:H/N];M_N)
\end{align*}
where $M_N$ is the group of coinvariants  of $M$.
\end{corollary}

\subsection{Takasu relative group homology}
Let $G$ be a discrete group and let $H$ be a subgroup of $G$. Regarding $EG$ as an $H$-CW--complex, we have a map $EH \to EG$ unique up to $H$--homotopy, which, finally, 
leads to a $G$-map $\iota_H^G\colon \ind_H^G EH \to EG$.  From now on, we will always assume that the map  $\iota_H^G$ is an inclusion, by replacing $EG$ with the mapping cylinder of $\iota_H^G$. This assumption is due to the fact that we want to consider the CW--pair $(EG, \ind_H^G EH )$, which does not always make sense. A second reason is that we will make use of some push-out constructions,  (see equation \eqref{eq:tak.po} and Theorem~\ref{LW:pushout}), hence our assumption will turn them into homotopy $G$-push-outs.

Consider an $\org$--module $M$. Define the \emph{Takasu relative group ho\-mo\-lo\-gy of $(G,H)$} to be
\begin{gather*}
H_*(G,H;M)\defeq H_*^G(EG, \ind_H^G EH ;M).
\end{gather*}

\begin{remark}
Note that:
\begin{itemize}
\item By Example~\ref{bredongmodules}, our definition reduces to the definition of Takasu relative group homology with coefficients in a trivial $G$--module given in \cite{AC17}.
\item If $H=I$ is the trivial subgroup, then we recover the classical homology of $G$ with coefficients in the $G$--module $M(G/I)$.
\item If $M$ is a $G$--module, then we recover the definition of Takasu homology with coefficients in a $G$--module given in \eqref{eq:Takasu.top.def} (see Example \ref{gmodulecoefficients} 
and Example \ref{localcoefficienthomology}). 
\end{itemize}
\end{remark}

Now we state the main property of Takasu relative group homology, this is, a long exact sequence that relates the homology of $G$ and $H$ to the homology of the pair $(G,H)$. This long exact 
sequence can be interpreted as the fact that Takasu's theory resembles the quotient of the homologies of $G$ and $H$. In fact, in Subsection \ref{products} we show an example where this phenomenon is more evident.

\begin{theorem}\label{takasulongexactsequence}
Let $G$ be a group and $H$ a subgroup. Let $M$ be a covariant $\org$--module. For $n\geq 0$, there exists a long exact sequence of the form, 
\begin{equation*}
\cdots \to \takasu{n+1}{G}{H}{M} \to H_n(H;\res_H^GM(H/I))\to H_n(G;M(G/I))\to \takasu{n}{G}{H}{M} \to \cdots.
\end{equation*} 
\end{theorem}

\begin{proof}
 It follows from the induction structure and the long exact sequence of the pair $(EG,\ind_H^G EH)$. In fact, we have the following commutative diagram where every vertical arrow 
 is an isomorphism using Lemma \ref{inductivestructureforbredon}, and Example \ref{localcoefficienthomology}
\[
\xymatrix{\cdots\ar[r] & H_{n}^{G}(\ind^G_H EH;M)\ar[r]\ar[d] & H_{n}^{G}(EG;M)\ar[r]\ar[d] & H_{n}^{G}(EG,\ind_{H}^{G}EH;M)\ar[r]\ar[d] & \cdots\\
\cdots\ar[r] & H_{n}^{H}(EH;\res_{H}^{G}M)\ar[r]\ar[d] & H_{n}^{G}(EG;M)\ar[r]\ar[d] & H_{n}(G,H;M)\ar[r]\ar[d] & \cdots\\
\cdots\ar[r] & H_{n}(H;\res_{H}^{G}M(H/I))\ar[r] & H_{n}(G;M(G/I))\ar[r] & H_{n}(G,H;M)\ar[r] & \cdots
}
\]
\end{proof}

\begin{remark}
 By Example \ref{bredongmodules}, this exact sequence reduces to the one in \cite{AC17} for $M$ the constant functor $\underline{\dbZ}$.
\end{remark}

Let us give an equivalent definition of Takasu relative homology, at least in degrees greater than or equal to $2$. This addresses Problem \ref{prob:cone}.

The \emph{Takasu space $T(G,H)$ of $(G,H)$} is the $G$-CW--complex  given by the following $G$--pushout
\begin{equation}\label{eq:tak.po}
\xymatrix{
\ind^G_H EH \ar[d] \ar[r]^{\iota_H^G} & EG \ar[d]  \\
G/H \ar[r]  & T(G,H),}
\end{equation}
where the left map is induced by collapsing each connected component of $\ind_H^G EH$ to a point. 

\begin{remark}\label{rem:tak.orb}
Note that the cone points form a $G$--orbit of $0$--cells of $T(G,H)$ that can be identified with $G/H$, hence we can make sense to the pair $(T(G,H),G/H)$. Note that the points in this orbit are the only ones with non-trivial isotropy.
\end{remark}

\begin{theorem}
Let $G$ be a group and $H$ a subgroup. Then, for all $n\geq0$, the quotient map $(EG,\ind_H^G EH)\to (T(G,H),G/H)$ induces an isomorphism
\[
H_n(G,H;M)\to H_n^G(T(G/H),G/H;M),
\]
for every $\Or(G)$--module $M$.
\end{theorem}
\begin{proof}
The proof is essentially the same as     that of \cite[Proposition 2.22]{Ha02}, using the excision and homotopy invariance axioms of Bredon homology.
\end{proof}

The following corollary gives an answer to Problem \ref{prob:cone}.

\begin{corollary}\label{absolute:homology:takasu}
For all $n\geq2$ and every $\Or(G)$--module $M$ (in particular for every $G$--module), We have the following isomorphism
\[
H_n(G,H;M)\cong H_n^G(T(G,H);M)
\]
\end{corollary}
\begin{proof}
It follows from the long exact sequence of the pair $(T(G,H),G/H)$ and the fact that $H_n^G(G/H;M)$ is zero for all $n\geq1$ and for every $M$.
\end{proof}

\subsection{The comparison homomorphism}
Since $T(G,H)$ is defined via the $G$-pushout \eqref{eq:tak.po}
and both $\ind^G_H EH$ and $EG$ are free $G$-CW--complexes, while for the points of $G/H$ all the isotropy groups are conjugated to $H$, we conclude that all the isotropy groups belong to the family $\calF (H)$. 
Therefore there is a $G$-map, unique up to $G$--homotopy 
\begin{equation}\label{comparison-spaces}
T(G,H) \to E_{\calF (H)} G,
\end{equation}
which leads to the $G$--map of pairs
\[
(T(G,H),G/H)\to (E_{\calF(H)}G,G/H).
\]
Applying the long exact sequence to this map of pairs, and using the fact that $H^G_i(G/H;M)=0$ for all $i\geq1 $, we get (see Corollary \ref{absolute:homology:takasu}), for $n\geq2$ the following commutative square
\begin{equation}\label{comparison:square}
\xymatrix{H_{n}(T(G,H);M)\ar[r]^{\cong}\ar[d] & H_{n}(G,H;M)\ar[d]\\
H_{n}([G:H];M)\ar[r]^{\cong} & H_{n}(E_{\calF(H)}G,G/H;M).
}
\end{equation}
Hence we have a homomorphism
\begin{equation*}
\varphi_n \colon \takasu{n}{G}{H}{M} \to \adamson{n}{G}{H}{M}, 
\end{equation*}
for all $n\geq2$ and for every covariant $\org$--module $M$, and analogously for cohomology. We call $\varphi$ the \emph{comparison homomorphism}.

\begin{remark}
Using the construction above, we still have the following commutative diagram with exact rows, that might help us to compare Takasu and Adamson relative homology theory in dimensions $0$ and $1$:
\[\scriptsize
\xymatrix{0\ar[r] & H_{1}(T(G,H);M)\ar[r]\ar[d] & H_{1}(G,H;M)\ar[r]\ar[d] & H_{0}(G/H;M)\ar[r]\ar[d]^{=} & H_{0}(T(G,H);M)\ar[r]\ar[d] & H_{0}(G,H;M)\ar[r]\ar[d] & 0\\
0\ar[r] & H_{1}([G:H];M)\ar[r] & H_{1}(E_{\calF(H)}G,G/H;M)\ar[r] & H_{0}(G/H;M)\ar[r] & H_{0}([G:H];M)\ar[r] & H_{0}(E_{\calF(H)}G,G/H;M)\ar[r] & 0
}
\]
Note that, in case $M$ is a constant $\Or(G)$--module, then the maps $H^G_0(G/H;M) \to H^G_0(T(G,H);M)$ and $H^G_0(G/H;M) \to H^G_0(E_{\calF(H)}G;M)$ are split injective since they can be identified with 
singular homology maps (see Example~\ref{bredongmodules}), therefore, at least in this case, we also have defined the comparison homomorphism in dimensions $0$ and $1$.
\end{remark}

Recall that a subgroup $H$ of $G$ is said to be \emph{malnormal} if $g^{-1} H g\cap H= I$ for all $g\in G\setminus H$. The following theorem addresses the first question in  Problem~\ref{problem:comparison homomorphism}.

\begin{theorem}\label{comparisontheorem}
The following are equivalent
\begin{enumerate}
\item $H$ is a malnormal subgroup of $G$;
\item the Takasu space $T(G,H)$ of $(G,H)$ is a model for the classifying space $E_{\calF(H)}G$.
\item The comparison homomorphism \[\varphi \colon \takasu{*}{G}{H}{M} \to \adamson{*}{G}{H}{M} \]
is an isomorphism for every covariant $\org$--module $M$;
\end{enumerate}
\end{theorem}

\begin{proof}
By \cite[Proposition 7.11]{AC17}, the Takasu space $T(G,H)$ is a model for $E_{\calF(H)}G$ if and only if $H$ is a malnormal subgroup of $G$. Hence (1) is equivalent to (2).

Assume that $T(G,H)$ is a model for $E_{\calF(H)}G$.  Since $E_{\calF(H)}G$ is unique up to $G$ homotopy, the (unique up to $G$ homotopy) $G$--map $T(G,H)\to E_{\calF(H)}G$ in \eqref{comparison-spaces} 
induces an isomorphism in homology for every $\org$--module $M$, therefore the left map in (\ref{comparison:square}) is an isomorphism and we conclude that the comparison homomorphism is an isomorphism.

Now suppose that $\varphi$ is an 
isomorphism for all $\org$--modules $M$, we shall verify that $T(G,H)$ is a model for $E_{\calF(H)}G$. 

From the definition, we can conclude that $T(G,H)$ is contractible since we are collapsing contractible subcomplexes of the contractible space  $EG$ to obtain $T(G,H)$. Also, we know that, for a non-trivial subgroup $K \in \calF(H)$, 
the space $T(G,H)^K$ coincides with $(G/H)^K$, in particular, it is a discrete subspace of $T(G,H)$. Hence $T(G,H)^K$ is contractible if and only if it consist of exactly one point. 
Now we shall prove that $T(G,H)^K$ is a discrete space with the (singular) cohomology of the one-point space. From \cite[Pag. 252, eq. (5)]{MS95}, there exists an $\org$--module $N$, 
such that, for any $G$--space $X$, we have the following natural isomorphism
\[ H_G^*(X; N)\cong H^*(X^K;\dbZ) \]
where the right hand side is singular cohomology. Hence $\varphi$ induces the following commutative diagram
\[
\xymatrix{
H_G^*(T(G,H);N)\ar[d]^{\varphi^*} \ar[r]^{\cong}  & H^*(T(G,H)^K;\dbZ)\ar[d]^{\varphi^*}\\
H_G^*(E_{\calF(H)}G;N) \ar[r]^{\cong } & H^*((E_{\calF(H)}G)^K;\dbZ).}
\]
Hence, we conclude that $T(G,H)^K$ is a one-point space. This finishes the proof.
\end{proof}

\begin{remark}
In \cite[Theorem 7.13]{AC17},  the \textit{if} part of the Theorem~\ref{comparisontheorem} is proved using a constant $\org$--module as coefficients, while for the \textit{only if} part we strongly use the 
fact that the comparison homomorphism is an isomorphisms for any coefficients.
\end{remark}

\section{Adamson and Takasu theories via derived functors}\label{sec:derived-functors}

Let $G$ be a group and let $\calF$ be a family of subgroups of $G$. In this section we recall the definition of the derived functor $\Tor[]{}$ in the category of $\Or(G,\calF)$--modules. Then, taking $H\leq G$, we use them to give algebraic definitions of Adamson and Takasu relative homology theories for the pair $(G,H)$. We then give a description of the comparison homomorphism using derived functors.

\vskip 5pt

Recall  there is an inclusion $\rmodu \Or(G,\calF)\longrightarrow \rmodu \Or(G)$, which sends a contravariant $\Or(G,\calF)$--module $M$ to a contravariant $\Or(G)$--module, which we denote by $\widehat{M}$, by setting $\widehat{M}(G/K)=M(G/K)$ for $K\in \calF$ and $\widehat{M}(G/K)=0$ for $K\notin \calF$.

\subsection{$\Tor[]{}$  functors}

A sequence of $\Or(G,\calF)$-modules is exact if, for all $K\in \calF$, when evaluated at $G/K$, the resulting sequence of abelian groups is exact. 

Let $N$ be a contravariant $\orgf$--module and let $M$ be a covariant $\orgf$--module. Consider a projective resolution $\CC[N]{P}$ of $N$. As in the classical setting we have a \emph{$\Tor[]{}$ functor} given by 
\begin{align*} 
\Tor[\Or(G,\calF)]{i}(N,M)&=H_i(\CC[N]{P}\otimes_{\Or(G,\calF)}M).
\end{align*}

\subsection{Adamson relative group homology}
Now we are interested in defining Adamson ho\-mo\-lo\-gy of a pair $(G,H)$ using the $\Tor[]{}$ and $\Ext[]{}$ functors. For this, we take $\calF=\calF(H)$. 

One can obtain $\Or(G,\calF(H))$--projective resolutions from a model of $E_{\calF(H)}G$ (see \cite[p.~11]{MV03}). We include this construction for completeness. First, define the \emph{augmentation homomorphism} $\varepsilon:\CC[0]{C}(X)(-)\to \underline{\dbZ}$ as the usual augmentation homomorphism $\varepsilon:\CC[0]{C}(X^K)\to \dbZ$ for all $K\in \calF(H)$.

\begin{proposition}\label{bredonresolution}
Let $X$ be a model for $E_{\calF(H)}G$. Then the augmented chain complex $\CC{C}(X)$ is a free, and therefore projective, $\Or(G,\calF(H))$--resolution of the constant $\Or(G,\calF(H))$-module $\underline{\dbZ}$.
\end{proposition}

\begin{proof}
Since $X^K$ is contractible, the augmented chain complex $\CC{C}(X)(G/K)=\CC{C}(X^K)$ is acyclic. Therefore the augmented chain complex functor $\CC{C}(X)$ is exact.

On the other hand, since $X$ has isotropy groups in $\calF(H)$, the $\orgf$--modules $\CC[i]{C}(X)$ are all free modules. Therefore the augmented chain complex $\CC{C}(X)$ is a free resolution of $\underline{\dbZ}$. 
\end{proof}

Let $\CC{P}$ be an $\Or(G,\calF(H))$--projective resolution of the constant $\Or(G,\calF(H))$--module $\underline{\dbZ}$. Hence the induced sequence $\widehat{\CC{P}}$ of $\Or(G)$--modules, is an $\Or(G)$--projective resolution of the $\Or(G)$--module $\widehat{\underline{\dbZ}}$.

\begin{theorem}\label{adamson-derivado}
Let $G$ be a group and $H$ a subgroup. For all $i\geq0$ and for all $\orgf$--module $M$, we have the following isomorphisms
\begin{gather*} 
H_i([G:H];M)\cong \Tor[\Or(G)]{i}(\widehat{\underline{\mathbb{Z}}},M)
\end{gather*}
\end{theorem}
\begin{proof}
It follows from Proposition~\ref{bredonresolution}.
\end{proof}

\begin{remark}
In \cite{Hochschild:RHA}, Adamson relative group homology is described as derived functors, using the language of relative homological algebra. On the other hand, in the above description, 
Adamson homology is described using derived functors within an abelian category with enough projectives. This gives an algebraic approach completely analogue to the classical group homology.
\end{remark}

\subsection{Takasu relative group homology}\label{takasuviaderived}
Now we want to define Takasu homology of a pair $(G,H)$ using derived functors.
In Takasu's original paper \cite{takasu-RHARCTOG}, there is already a very nice approach for $H_*(G,H;M)$ via derived functors (see also \cite[Section 5.3]{AC17}).
Recall that the category of contravariant $\Or(G,\triv)$--modules is canonically isomorphic to both, the category of 
right $G$--modules and the category of left $G$--modules, hence there is no substantial difference between the three of them (see Remark~\ref{naturality-modules}). Consider the $\Or(G,\triv)$--module $I_{(G,H)}(\dbZ)\defeq \ker(\dbZ[G/H]\to \dbZ)$ and the corresponding $\Or(G)$--module $\widehat{I_{(G,H)}(\dbZ)}$. Hence we have that $\widehat{I_{(G,H)}(\dbZ)}(G/I)=I_{(G,H)}(\dbZ)$ and $\widehat{I_{(G,H)}(\dbZ)}(G/K)=0$ for all $K\neq I$. Hence, a $G$--projective resolution of the $G$--module $I_{(G,H)}(\dbZ)$ corresponds to a $\Or(G)$--projective resolution of the $\Or(G)$--module $\widehat{I_{(G,H)}(\dbZ)}$.

\begin{theorem}
For every $\Or(G)$--module $M$ and every $i\geq 1$, We have the following isomorphisms
\begin{gather*} 
H_i(G,H;M)\cong \Tor[G]{i-1}(I_{(G.H)}(\dbZ),M(G/I))\cong \Tor[\Or(G)]{i-1}(\widehat{I_{(G,H)}(\dbZ)},M).
\end{gather*}
\end{theorem}
\begin{proof}
Consider the standard resolution $\CC{C}(G)$ (resp. $\CC{C}(H)$) of the trivial $G$--module (resp. $H$--module) $\dbZ$. By \cite[Proposition~3.3]{takasu-RHARCTOG} the exact sequence
\begin{equation}\label{eq:tak.res}
\xymatrix{ \cdots \ar[r]& \CC[2]{C}(G)/\ind_{H}^{G}\CC[2]{C}(H)\ar[r]& \CC[1]{C}(G)/\ind_{H}^{G}\CC[1]{C}(H)\ar[r]& I_{(G,H)}(\dbZ)\ar[r]&0.}
\end{equation}
is a $G$--projective resolution of $I_{(G,H)}(\dbZ)$, so it can be seen as an $\Or(G)$--projective resolution of $\widehat{I_{(G,H)}(\dbZ)}$. On the other hand, by Remark~\ref{rem:tak.orb}, the sequence \eqref{eq:tak.res} corresponds to the cellular chain complex of the Takasu space $T(G,H)$. Thus for any $\Or(G)$--modulo $M$, \eqref{eq:tak.res} can be used to compute the 
Takasu relative homology of $(G,H)$.
\end{proof}

\subsection{The algebraic version of the comparison homomorphism}\label{ssec:alg.com.hom}

As in Subsection~\ref{takasuviaderived}, consider an $\Or(G)$--projective resolution $\CC{P}\to \widehat{I_{(G,H)}(\dbZ)}$.

On the other hand, since the family $\calF(H)$ is generated by $H$, we have that the augmentation map $\epsilon \colon \dbZ[-,G/H] \to \underline{\dbZ}$ is surjective. Next complete this surjection to a projective $\Or(G,\calF(H))$--resolution $\CC{Q}\to \underline{\dbZ}$, so that $\CC[0]{Q}=\dbZ[-,G/H]$. 
Notice that $I_{(G,H)}(\dbZ)=\ker{\varepsilon}(G/I)$. Consider an $\Or(G)$--projective resolution $\CC{\widehat{Q}}\to \underline{\hat \dbZ}$. Since $I_{(G,H)}(\dbZ)=\ker{\varepsilon}(G/I)$, the inclusion $\widehat{I_{(G,H)}(\dbZ)}\to \widehat{\ker{\varepsilon}}$, can be extended to a morphism of resolutions 
\begin{equation}\label{alg-morphism}
\widetilde{\varphi_*}:\CC{P} \to \CC{\widehat{Q}}.
\end{equation}
Given an  $\Or(G)$--module $M$, this induces homomorphisms
\begin{equation}\label{eq:comp.hom}
\varphi_i \colon \takasu{i}{G}{H}{M} \to \adamson{i}{G}{H}{M},\quad\text{for all $i\geq 2$.}
\end{equation}

Finally, we construct a long exact sequence on which the comparison homomorphism (for each dimension) fits. This addresses the second question in Problem~\ref{problem:comparison homomorphism}.

Let $J(G,H)$ be the quotient $\ker{\varepsilon}/I_{(G,H)}(\dbZ)$, so that we have the following long exact sequence
\[0\to I_{(G,H)}(\dbZ) \to \ker{\varepsilon} \to J(G,H) \to 0\]
Note that $\Tor[\Or(G)]{i-1}(\ker \varepsilon,M)=\adamson{i}{G}{H}{M}$. Given an $\org$--module $M$, we get  a long exact sequence of $\Tor[\Or(G)]{*}(-,M)$ groups

\begin{multline*}
\cdots \to \takasu{i}{G}{H}{M} \xrightarrow{\varphi_i} \adamson{i}{G}{H}{M} \to \Tor[\Or(G)]{i-1}(\widehat{J_{(G,H)}},M)\to  \takasu{i-1}{G}{H}{M}\to \cdots \\
\to \takasu{2}{G}{H}{M} \xrightarrow{\varphi_2} \adamson{2}{G}{H}{M}.
\end{multline*}

Essentially $\Tor[\Or(G)]{*}(\widehat{J_{(G,H)}},M)$ is measuring how far is the comparision homomorphism from being an isomorphism. 

We describe $J_{(G,H)}=\ker{\varepsilon}/I_{(G,H)}(\dbZ)$ in more detail. Recall that $\varepsilon\colon \dbZ[-,G/H] \to \underline{\dbZ}$ is the augmentation homomorphism. Then we have
\[
J_{(G,H)}(G/K)=\begin{cases}
\ker(\dbZ[(G/H)^K]\to \dbZ) & \text{ if }K\neq I \text{ and }K\in \calF(H),\\
0 & \text{otherwise}. \end{cases}
\]

The fixed point set $(G/H)^K$ is the set of all cosets $g_iH$ such that $K\leq g_i^{-1}Hg_i$. In particular, if $K\in \calF(H)$ is contained in only one conjugate $g_i^{-1}Hg_i$, then $\ker(\dbZ[(G/H)^K]\to \dbZ)=\ker(\dbZ\to \dbZ)=0$. Hence, we can conclude that $J(G,H)$, and therefore $\Tor[\Or(G)]{*}(\widehat{J_{(G,H)}},M)$,  only depends on the following family
\[
\calG(H)\defeq\{ K\leq G : K\leq  g_i^{-1}Hg_i \cap g_j^{-1}Hg_j \text{ for some } g_i,g_j\text{ such that }g_iH\neq g_jH \}\subset \calF(H)
\]
 Explicitely, for $K\in \calF(H)$, 
 \[
J_{(G,H)}(G/K)=\begin{cases}
\ker(\dbZ[(G/H)^K]\to \dbZ) & \text{ if }K\neq I \text{ and }K\in \calG(H),\\
0 & \text{otherwise}\end{cases}
\]

As a very particular case, assume $H$ is a malnormal subgroup of $G$. Then $\calG(H)$ is the trivial family, hence  $J(G,H)(G/K)=0$ for every $K\in \calF(H)$. Thus we partially recover Theorem~\ref{comparisontheorem}, that is, we have that the comparison homomorphism $\varphi_i$ is an isomorphism, for $i\geq2$, provided $H$ is a malnormal subgroup of $G$. 

\section{The L\"uck--Weiermann construction and some exact sequences}\label{LWconstruction}

In this section, we obtain in Corollary \ref{MayerVietoris}, a Takasu-type long exact sequence  for Adamson homology. To get such a sequence we describe a construction of the classifying space $E_{\calF(H)}G$ similar to \eqref{eq:tak.po}, using \cite[Theorem~2.3]{LW12}. 

First we need some notation. Let $G$ be a group. Consider two families $\calF\subseteq \calG$ of subgroups of $G$. 
Suppose we have an equivalence relation $\sim$ on $\calG \setminus \calF$ satisfying:

\begin{itemize}
\item \textbf{Closed under taking subgroups:} For $H,K\in \calG \setminus \calF$, with $H\subseteq K$ then must be $H\sim K$, and
\item \textbf{Invariant under conjugation:} For $H,K\in \calG \setminus \calF$ and $g\in G$, then must be $H\sim K$ if and only if $g^{-1}Hg \sim g^{-1} Kg$.
\end{itemize}

For $H\in \calG \setminus \calF$, define the subgroup of $G$
\begin{equation*}
N_G[H] = \{ g\in G ~|~ g^{-1} Hg\sim H \}.
\end{equation*}
Also define the following family of subgroups of $N_G[H]$
\[ \calG[H] = \{ K\subseteq N_G[H] ~|~  K\in \calG \setminus \calF, K\sim H \} \cup \{ K\subseteq N_G[H] ~|~ K\in \calF \}. \]

As an immediate consequence of \cite[Theorem 2.3]{LW12} we have the following,

\begin{theorem}\label{LW:pushout}
With the same notation as above,
let $H$ be a subgroup of $G$, let $\calG$ the family of subgroups of $G$ generated by $H$, and let $\calF$ be a subfamily of $\calG$. Choose models for $E_{\calF}N_G[H]$ and $E_{\calG[H]}N_G[H]$, 
and a model for $E_{\calF}G$. Now consider $X$ defined by the $G$-pushout:
	 $$ \xymatrix{ G\times_{N_G[H]}E_{\calF}N_G[H] \ar[r]^-i \ar[d]^{\id_G\times_{N_G[H]}f_{H}} & E_{\calF}G \ar[d] \\ G\times_{N_G[H]} E_{\calG[H]}N_G[H] \ar[r] & X}$$
	 where the maps starting from the left upper corner are cellular and one of them is an inclusion of $G$-CW-complexes. Then $X$ is a model for $E_{\calG}G$.
\end{theorem}

\begin{remark}
The $G$--pushout from the previous theorem, can be seen as a kind of analogue of the $G$-pushout used to define the Takasu's space $T(G,H)$. This point of view might be helpful since this pushout 
construction gives us a more explicit description of the difference between the classifying space $E_{\calF(H)}G$ and Takasu's space $T(G,H)$.
\end{remark}

\begin{corollary}\label{MayerVietoris}
Let $G$ be a group, let $\calG$ be the family generated by a subgroup $H$ of $G$, let $\calF$ be a subfamily of $\calG$ and let $M$ be an $\Or(G,\calG)$--module. Then we have the following Mayer--Vietoris exact sequence
\begin{multline*} 
\scriptsize
\cdots \to H_n^{N_G[H]}(E_{\calF}N_G[H];\res^G_{N_G[H]}M) \to H_n^{G}(E_{\calF} G;M) \oplus  H_n^{N_G[H]}(E_{\calG[H]} N_G[H];\res^G_{N_G[H]}M)\\ \to H_n([G:H];M)\to \cdots 
\end{multline*}
In particular, if $\calF$ is the trivial family, then we have the following exact sequence
\begin{multline*}
\cdots \to H_n(N_G[H];\res^G_{N_G[H]}M(N_G[H]/I)) \to \\ H_n(G;M(G/I)) \oplus  H_n^{N_G[H]}(E_{\calG[H]} N_G[H];\res^G_{N_G[H]}M(N_G[H]/I)  \to H_n([G:H];M)\to \cdots 
\end{multline*}
\end{corollary}

\subsection{Malnormal subgroups}
Let $H$ be a malnormal subgroup of $G$. We consider $\triv \subset\calF(H)$. Define the equivalence relation $\sim$ on $\calF(H)\setminus \triv$ as follows: Let $K_1$ and $K_2$ be groups in $\calF(H)\setminus \triv$, 
then there exists $g_1, g_2\in G$ such that $K_1\subseteq g_1H{g_1}^{-1}$ and $K_2\subseteq g_2H{g_2}^{-1}$. Hence we say that $K_1\sim K_2$ if and only if $g_{1}H{g_1}^{-1}\cap g_2H{g_2}^{-1}\neq I$. Then $N_G[H]=\{ g\in G ~|~ gHg^{-1} \cap H \neq I \}=H$ and $\calG[H]$ is the family of all subgroups of $H$.

Consider an $\Or(G)$--module $M$, then applying Corollary~\ref{MayerVietoris} we have the following long exact sequence
\begin{equation*}
\cdots \to H_n^H(E_{\triv}H;M)\to H_n(G;M) \oplus H_n^B(E_{\all}H;M)\to H_n([G:H];M)\to \cdots
\end{equation*}
since $E_{\all}B$ has as a model the one-point space. Hence the above sequence translates to
\begin{equation*}
\cdots \to H_{n+1}([G:H];M) \to H_n(H;M)\to H_n(G;M) \to H_n([G:H];M)\to \cdots
\end{equation*}

\subsection{The long exact sequence for good triples.}
In this subsection we will show that, for certain triples $K\leq H\leq G$, we have a long exact sequence in Adamson homology similar to the sequence of a triple of spaces in singular homology.

Consider a triple of groups $K\leq H\leq G$, and consider the families $\calF(K)$ and $\calF(H)$ generated by $K$ and $H$ respectively. Define the following equivalence relation in 
$\calF(H)\setminus \calF(K)$: for any $H_1$, $H_2$ in $\calF(H)\setminus \calF(K)$, there exist $g_1, g_2 \in G$ such that $H_1 \leq g_1H\inv{g_1}$ and $H_2 \leq g_2H\inv{g_2}$, then we set
\[
H_1\sim H_2 \Longleftrightarrow g_1H\inv{g_1} \cap g_2H\inv{g_2} \in \calF(H) \setminus \calF(K),
\]
i.e. $g_1H\inv{g_1} \cap g_2H\inv{g_2}$ is not subconjugate to $K$.

It is straightforward that this equivalence relation satisfies the required properties in order to use the L\"uck--Weiermann construction.

We say that $K\leq H \leq G$ is \emph{a good triple}, if $H=N_G[H]=\{g\in G ~|~ H\sim gH\inv{g} \}$, i.e. if, for all $g\in G\setminus H$, $H\cap gH\inv{g}$ is subconjugated to $K$.

\begin{theorem}
Using the above notation. Suppose that $K\leq H \leq G$ is a good triple. Then, for all $\Or(G)$--module $M$, we have a long exact sequence such that, for all $n>0$, looks like
\[
\cdots \to \adamson{n}{H}{K}{\res_H^G M} \to \adamson{n}{G}{K}{M} \to \adamson{n}{G}{H}{M} \to  \adamson{n-1}{H}{K}{\res_H^G M} \to \cdots,
\]
while for $n=0$ we have
\[
\cdots \to \adamson{0}{H}{K}{\res_H^G M} \to \adamson{0}{G}{K}{M} \oplus M(G/H) \to \adamson{0}{G}{H}{M} \to 0.
\]
\end{theorem}
\begin{proof}
Note that for a good triple $N_G[H]=H$ and $\calG[H]$ is the family of all subgroups of $N_G[H]=H$, since $H$ always belongs to $\calG[H]$ by definition. Hence $E_{\calG[H]}N_G[H]$ has as a model the one-point space.

The conclusion follows from Corollary \ref{MayerVietoris} by setting $\calF=\calF(K)$, $\calG=\calF(H)$.
\end{proof}

\begin{example}
If $H$ is a malnormal subgroup of $G$, then it follows that for every $K\leq H$, the triple $K\leq H \leq G$ is a good triple. 
\end{example}

\begin{example}
Consider the triple $T\leq B \leq G$, where $G=SL_2(\dbC)$, 
\begin{equation*}
B=\left\{ \left(\begin{matrix}a & b\\
0 & a^{-1}\end{matrix}\right) ~\middle|~ a\in\dbC^*,b\in \dbC  \right\}, \text{ and } T=\left\{ \left(\begin{matrix}a & 0\\
0 & a^{-1}\end{matrix}\right) ~\middle|~ a\in\dbC^*  \right\}.
\end{equation*}
We also consider this groups endowed with the discrete topology. Then, we claim that they are a good triple. In fact, since $B$ is a proper maximal subgroup of $G$ 
(see \cite[Proposition XIII.8.2]{La02}), we conclude that $N_G[B]$ is either $G$ or $B$. Therefore it suffices to exhibit an element in $G$ that it does not belong to $N_G[B]$, i.e. 
an element $g\in G$ such that $gB\inv{g}\cap B$ is subconjugate to $T$. A direct computation shows that we can take $g=\begin{pmatrix}0 & 1\\
-1 & 0\end{pmatrix}$.
\end{example}

\section{Examples}\label{sec:exa}

\subsection{Normal subgroups}

Let $G$ be a group and $H$ a normal subgroup of $G$. Let $M$ be a $G$--module. Then by Corollary \ref{adamsongmodules}
\[
\adamson{*}{G}{H}{M} \cong H_*(G/H;M_H)
\]
Recall the Lyndon--Hochshild--Serre spectral sequence
\[
E^2_{p,q}= H_p(G/H;H_q(H;\res_H^GM)) \Longrightarrow H_{p+q} (G;M)
\]

\begin{proposition}\label{normalsubgroup}
Suppose that we have  a normal subgroup $H$ of $G$,  and a $G$-module $M$ such that $E^2_{p,q}=0$, for all $p,q>0$. Then we have the commutative diagram

\[\scriptsize
\xymatrix{\cdots\ar[r] & H_{n}(H;\res_{H}^{G}M)_{G/H}\ar[r] & H_{n}(G;M)\ar[r] & H_{n}(G/H;M_{H})\ar[r] & H_{n-1}^ {}(H;\res_{H}^{G}M)_{G/H}\ar[r] & \cdots\\
\cdots\ar[r] & H_{n}(H;\res_{H}^{G}M)\ar[r]\ar[u] & H_{n}(G;M)\ar[r]\ar@{=}[u] & H_{n}(G,H;M)\ar[r]\ar[u] & H_{n-1}(H;\res_{H}^{G}M)\ar[r]\ar[u] & \cdots
}
\]
where the first vertical arrow is the quotient map, the second is the identity, and the third is the comparison map composed with the isomorphism above.
\end{proposition}
\begin{proof}
The top row is straightforward from the spectral sequence and the hypothesis $E^2_{p,q}=0$, for all $p,q>0$. 
The commutativity of the diagram follows from the following commutative diagram (up to $G$--homotopy)
\[
\xymatrix{G\times_{H}EH\ar[d]^{=}\ar[r] & EG\ar[d]^{=}\ar[r] & T(G,H)\ar[d]\\
G\times_{H}EH\ar[r] & EG\ar[r] & E_{\mathcal{F}(H)}G=E(G/H)
}
\]
and the definition of the transgression (see \cite[p. 185]{Mc01}) homomorphism in the Lyndon--Hochshild--Serre spectral sequence.
\end{proof}

\begin{corollary}
Consider $G$, $H$, and $M$ as in the previous proposition. If the action of $G$ on $H$ is by inner automorphisms, then the comparison homomorphism
\[\varphi \colon  \takasu{*}{G}{H}{M} \to \adamson{*}{G}{H}{M} \]
is an isomorphism, for all $G$--modules $M$.
\end{corollary}
\begin{proof}
It follows from the previous proposition and the five lemma, since the left vertical arrow in the Proposition~\ref{normalsubgroup} is then the identity.
\end{proof}

\begin{remark}
Note that the hypothesis from Proposition \ref{normalsubgroup} and its corollary depend on the coefficients $M$. Therefore this does not lead to a contradiction with Theorem \ref{comparisontheorem}.
\end{remark}

\begin{example}
As a straightforward example we have the cyclic groups $(C_n,C_m)$ with $m$ and $n$ relative primes and $M$ equal to the trivial $G$--module $\dbZ$. In fact, using the universal 
coefficient theorem and the fact that $H_i(C_m;\dbZ)$ vanishes for $i$ odd, we have that 
\[
E^2_{p,q}= H_p(C_n/C_m;H_q(C_m;M))\cong 0
\]
for all $p,q>0$. Moreover, since both $\adamson{i}{C_n}{C_m}{\dbZ}$ and $\takasu{i}{C_n}{C_m}{\dbZ}$ are finite groups for every $i>0$, we can conclude that the comparison 
homomorphism is actually an isomorphism. We will show in Subsection \ref{cyclegroups} below that, the hypothesis of $m$ and $n$ being relative prime numbers, is necessary.
\end{example}

\subsection{The pair $(K\times H,H)$}\label{products}
Throughout this section we assume the coefficients module is $\underline{\dbZ}$. Consider the pair $(K\times H,H)$, where we are making an abuse of 
notation by denoting $I\times H$ by $H$. Then we have the split short exact sequence
\[
I\to H\underset{j}{\overset{i}{\rightleftarrows}} K\times H \to K \to I,
\]
and the projection homomorphism $j\colon K\times H\to H$ such that $j\circ i$ is the identity homomorphism on $H$.

Now, from Proposition \ref{takasulongexactsequence} we have the following long exact sequence for Takasu's relative homology
\[
\cdots \to H_{n+1}(K\times H,H) \to H_n(H) \xrightarrow{i_*} H_n(K\times H) \to H_n(K\times H,H) \to \cdots
\]
Thus, the homomorphism induced by $j$ leads to short split exact sequences
 \[
I\to H_n(H) \underset{j_*}{\rightleftarrows} H_n(K\times H) \to H_n(K\times H,H) \to I
 \]
 for all $n\geq0$. Therefore 
 \[
 H_n(K\times H,H) \cong H_n(K\times H)/H_n(H)
 \]
 for all $n\geq 1$.

On the other hand, $H_n(G)=H_n(K\times H)$ can be calculated using the K\"unneth formula
\[
H_n(K\times H)\cong \left( \bigoplus_{i+j=n}H_i(K)\otimes H_j(H)\right)\oplus \left( \bigoplus_{i+j=n-1} \Tor{}(H_{i}(K),H_j(H))\right).
\]

While for Adamson's group homology, from Proposition \ref{prop:3.IT}, we have
\[
H_i([K\times H,H])\cong H_i(K\times H /H) = H_i(K)
\]

In this case the comparison homomorphism $\varphi_i\colon H_i(K\times H,H) \to H_i ([K\times H:H])$ can be completely described. In fact, it is not difficult to see that this 
homomorphism comes from projecting $H_*(G,H)$ onto the copy of $H_*(K)$ contained as a summand (described above as the column in the spectral sequence-type array). 
Hence the comparison homomorphism is surjective and the kernel is given by
\[
\bigoplus_{\substack{i+j=n \\ i>0, j>0} }H_i(K;H_j(H)) \cong
\left( \bigoplus_{\substack{i+j=n \\ i>0, j>0} }H_i(K)\otimes H_j(H)\right)\oplus \left( \bigoplus_{\substack{i+j=n-1 \\ i>0, j>0} } \Tor{}(H_{i}(K),H_j(H))\right)
\]

For the case of more general coefficients an explicit description of the comparison homomorphism, seems to be more complicated.

\subsection{Finite cyclic groups}\label{cyclegroups}
Let $C_i$ denote the cyclic group of order $i$. In \cite[Example 7.3]{AC17}, the authors exhibit the pair $(C_4,C_2)$ as an example where the Adamson's and Takasu's homology groups are isomorphic, 
which might sound as a contradiction with Theorem \ref{comparisontheorem}. Nevertheless, they do not mention anything about this isomorphism being related to the comparison homomorphism. 
In this example we show that the comparison homomorphism $\varphi_i\colon \takasu{i}{C_4}{C_2}{\dbZ} \to \adamson{i}{C_4}{C_2}{\dbZ} $ vanishes for for all $i\geq2$, in particular, it is not an isomorphism.

Let $G=C_n$ be the cyclic group of order $n$ generated by the element $t$. Consider  $N=1+t+t^2+\cdots+ t^{n-1}$, then
\begin{equation}
\xymatrix{\nu:\cdots\ar[r]^{N}&\dbZ G \ar[r]^{t-1}&\dbZ G\ar[r]^{N}&\dbZ G\ar[r]^{t-1}& \dbZ G\ar[r]^{\varepsilon} &\dbZ\ar[r]&0}
\end{equation}
is a $G$--projective resolution of the trivial $G$--module $\dbZ$, where the homomorphisms are multiplication by $t-1$ and $N$, and  $\varepsilon$ is the augmentation homomorphism (see for instance \cite[Pag. 20]{Brown-COG}).

On the other hand, consider a subgroup $C_m=H$ of $G$. Hence, $H$ is normal and the  quotient group $G/H$ is also cyclic generated by $tH$, so  
\begin{equation}\label{resolution-c4}
\xymatrix{\nu':\cdots\ar[r]^{N'}&\dbZ[G/H] \ar[r]^{(t-1)H}&\dbZ [G/H]\ar[r]^{N'}&\dbZ [G/H]\ar[r]^{(t-1)H}& \dbZ [G/H]\ar[r]^{\varepsilon} &\dbZ\ar[r]&0},
\end{equation}
where $N'=H+tH+\cdots +t^{n/m-1}H$,  is a $G/H$--projective resolution of $\dbZ$.

From Section \ref{takasuviaderived} and  \cite[Lemma 5.3]{AC17}, we have the following facts,
\begin{itemize}
\item If $P$ is a projective $G$--module, then $I_{(G,H)}(P)$ is also a projective $G$--module;
\item The $G$--homomorphism $\eta\colon\dbZ G\otimes_{\dbZ H}\dbZ \to \dbZ[G/H]$, given by $\eta(g\otimes n)=ngH$, is an isomorphism;
\item Moreover $\eta$ restricts to an isomorphism $I_{(G,H)}(\dbZ)\cong \ker{(\varepsilon:\dbZ[G/H]\to \dbZ)}$;
\item $I_{(G,H)}(M)$ is generated (as $\dbZ$--module) by the elements $x\otimes n- 1\otimes xn$ for all $x\notin H$ and $n\in M$.
\end{itemize}

From now on we will focus on the case $(G,H)=(C_4,C_2)$. We can apply the functor $I_{(G,H)}(-)$ to the resolution $\nu$ in order to obtain a $G$--projective resolution of $I_{(G,H)}(\dbZ)$. 
Since $\nu'$ is an exact sequence, the isomorphism  $\eta\colon I_{(G,H)}(\dbZ)\to \ker{(\varepsilon)}$ can be extended to a morphism between the following exact sequences.

\begin{equation*}
\xymatrix{I_{(G,H)}(\nu):\cdots\ar[r]^{\id\otimes(t-1)} & I_{(G,H)}(\dbZ G)\ar[d]^{\lambda_2}\ar[r]^{\id\otimes N} & I_{(G,H)}(\dbZ G)\ar[d]^{\lambda_1}\ar[r]^{\id\otimes(t-1)} & I_{(G,H)}(\dbZ G)\ar[d]^{\lambda_0}\ar[r]^{\id\otimes \varepsilon} & I_{(G,H)}(\dbZ)\ar[d]^{\eta}\ar[r] & 0\\
\nu'\cdots\ar[r]^{(t+1)H} & \dbZ [G/H]\ar[r]^{(t-1)H} & \dbZ [G/H]\ar[r]^{(t+1)H} & \dbZ [G/H]\ar[r]^{(t-1)H} & \ker(\varepsilon)\ar[r] & 0.}
\end{equation*}

Our following task is to compute $\lambda_i$ for all $i\geq 0$. First we should note that $I_{(G,H)}(\dbZ G)$ is a rank-one free $G$--module with basis $t\otimes 1 -1 \otimes t$. In fact, 
by definition of $I_{(G,H)}(\dbZ G)$ we have the short exact sequence 
\begin{equation*}
0\to I_{(G,H)}(\dbZ G) \to \dbZ[G]\otimes_{\dbZ[H]} \dbZ[G] \to \dbZ[G] \to 0,
\end{equation*}
which splits since $\dbZ[G]$ is a free $G$--module. Therefore $I_{(G,H)}(\dbZ[G])$ is a free abelian group of rank $4$. Now, a direct calculation shows that $t^i(t\otimes 1 - 1\otimes t)$, 
for $i=0,1,2,3$ form a $\dbZ$--basis of $I_{(G,H)}(\dbZ G)$, and we conclude that it is a free $G$--module.

In order to calculate $\lambda_0$, we use the right hand side square of the above diagram:
\begin{align*}
\eta\circ(\id\otimes\varepsilon)(t\otimes1-1\otimes t) & =\eta(t\otimes1-1\otimes1)\\
 & =tH-H=(t-1)H,
\end{align*}
hence we can define $\lambda_0(t\otimes1-1\otimes t)=H$ and extend to a $G$--equivariant map. Now we compute $\lambda_1$ using the second square from right to left.
\begin{align*}
\lambda_{0}\circ(\id\otimes(t-1))(t\otimes1-1\otimes t) & =\lambda_{0}(t\otimes(t-1)-1\otimes(t^{2}-t))\\
 & =\lambda_{0}(t\otimes t-t\otimes1-1\otimes t^{2}+1\otimes t)\\
 & =\lambda_{0}(-(t\otimes1-1\otimes t)-t(t\otimes1-1\otimes t))\\
 & =\lambda_{0}(-(t+1)(t\otimes1-1\otimes t))\\
 & =-(t+1)H,
\end{align*}
therefore we can define $\lambda_1(t\otimes1-1\otimes t)=-H$.
For $\lambda_2$ we have

\begin{align*}
\lambda_{1}\circ(\id\otimes N)(t\otimes1-1\otimes t) & =\lambda_{1}(t\otimes(1+t+t^{2}+t^{3})-1\otimes(1+t+t^{2}+t^{3}))\\
 & =\lambda_{1}((t\otimes1-1\otimes) t-t(t\otimes1-1\otimes t)+t^{2}(t\otimes1-1\otimes t)-t^{3}(t\otimes1-1\otimes t))\\
 & =\lambda_{1}((1-t+t^{2}-t^{3})(t\otimes1-1\otimes t))\\
 & =2(t-1)H
\end{align*}
and we define $\lambda_2(t\otimes1-1\otimes t)=2H$. Proceeding in a similar way we finally obtain the following formulas
\begin{gather*}
\lambda_{2i}(t\otimes1-1\otimes t)=2^iH,\\
\lambda_{2i+1}(t\otimes1-1\otimes t)=-2^iH.
\end{gather*}

And we can conclude that after tensoring with the trivial $G$--module $\dbZ$, we obtain the homomorphisms of abelian groups 
$\lambda_j\otimes Id\colon I_{(G,H)}(\dbZ) \otimes_{\dbZ G} \dbZ \to \dbZ[G/H] \otimes_{\dbZ G} \dbZ$ given by the following formulas
\begin{gather*}
\lambda_{2i}\otimes \id(t\otimes1-1\otimes t)=2^i,\\
\lambda_{2i+1}(t\otimes1-1\otimes t)=-2^i.
\end{gather*}
Finally the comparison homomorphisms at the level of chain complexes looks as follows
\begin{equation}\label{eq:c4.c2}
\xymatrix{\cdots\ar[r] & \dbZ\ar[r]^{0}\ar[d]^{4} & \dbZ\ar[r]^{2}\ar[d]^{2} & \dbZ\ar[r]^{0}\ar[d]^{2} & \dbZ\ar[r]^{2}\ar[d]^{1} & \dbZ\ar[r]\ar[d]^{1} & 0 \\
\cdots\ar[r] & \dbZ\ar[r]^{0} & \dbZ\ar[r]^{2} & \dbZ\ar[r]^{0} & \dbZ\ar[r]^{2} & \dbZ\ar[r] & 0}
\end{equation}
Now we can explicitly describe the comparison homomorphism $\varphi_i\colon \takasu{i}{C_4}{C_2}{\dbZ} \to \adamson{i}{C_4}{C_2}{\dbZ} $. In fact $\varphi_1$ is the identity while $\varphi_i=0$ for all $i\geq2$.

\begin{remark}
In general, the algebraic version of the comparison homomorphism only gives the comparison homomorphisms $\varphi_i\colon \takasu{i}{G}{H}{M} \to \adamson{i}{G}{H}{M}$ 
for $i\geq2$ (see Subsection~\ref{ssec:alg.com.hom} or \cite[\S7.1]{AC17}) because when we remove $\ker{\varepsilon}$ in the projective resolution to get 
the reduced resolution, the kernel becomes $\CC[1]{C}(G/H)$ but in general it is smaller, so this truncated sequence, does not compute $\adamson{1}{G}{H}{M}$. But in the 
particular case of this example, that does not happen because taking the reduced resolution \eqref{resolution-c4} tensored with the trivial coefficients $\dbZ$, we get that the homomorphism 
induced by the homomorphism before the augmentation is zero (see for instance \cite[Pag.~35]{Brown-COG}), so the lower sequence in diagram \eqref{eq:c4.c2} at the first place on the right, indeed computes $\adamson{1}{C_4}{C_2}{\dbZ}$, so we also get the comparison homomorphism $\varphi_1$.
\end{remark}

\bibliographystyle{alpha} 
\bibliography{newbib}

\begin{thebibliography}{ANCM18}

\bibitem[Ada54]{Adamson:CTNNSNNF}
Iain~T. Adamson.
\newblock {Cohomology theory for non-normal subgroups and non-normal fields}.
\newblock {\em Proc. Glasgow Math. Assoc.}, 2:66–76, 1954.

\bibitem[ANCM17]{AC17}
Jos\'e~Antonio Arciniega-Nev\'arez and Jos\'e~Luis Cisneros-Molina.
\newblock Comparison of relative group (co)homologies.
\newblock {\em Bol. Soc. Mat. Mex. (3)}, 23(1):41--74, 2017.

\bibitem[ANCM18]{Arciniega-Cisneros:IH3MRGH}
José~Antonio Arciniega-Nevárez and José~Luis Cisneros-Molina.
\newblock Invariants of hyperbolic 3-manifolds in relative group homology.
\newblock arXiv:1303.2986v3 [math.GT], November 2018.

\bibitem[Blo77]{Blowers:CSPR}
James~V. Blowers.
\newblock {The classifying space of a permutation representation}.
\newblock {\em Trans. Amer. Math. Soc.}, 227:345--355, 1977.

\bibitem[Bre67]{Bredon:ECT}
Glen~E. Bredon.
\newblock {\em Equivariant cohomology theories}.
\newblock Lecture Notes in Mathematics, No. 34. Springer-Verlag, Berlin-New
  York, 1967.

\bibitem[Bro82]{Brown-COG}
Kenneth~S. Brown.
\newblock {\em {Cohomology of groups}}, volume~87 of {\em {Graduate Texts in
  Mathematics}}.
\newblock Springer-Verlag, New York, 1982.

\bibitem[DL98]{DL98}
James~F. Davis and Wolfgang L{\"u}ck.
\newblock Spaces over a category and assembly maps in isomorphism conjectures
  in {$K$}- and {$L$}-theory.
\newblock {\em $K$-Theory}, 15(3):201--252, 1998.

\bibitem[Hat02]{Ha02}
Allen Hatcher.
\newblock {\em Algebraic topology}.
\newblock Cambridge University Press, Cambridge, 2002.

\bibitem[Hoc56]{Hochschild:RHA}
G.~Hochschild.
\newblock {Relative homological algebra}.
\newblock {\em Trans. Amer. Math. Soc.}, 82:246–269, 1956.

\bibitem[IK14]{Inoue-Kabaya:QHCV}
Ayumu Inoue and Yuichi Kabaya.
\newblock {Quandle homology and complex volume}.
\newblock {\em Geom. Dedicata}, 171:265–292, 2014.

\bibitem[Lan02]{La02}
Serge Lang.
\newblock {\em Algebra}, volume 211 of {\em Graduate Texts in Mathematics}.
\newblock Springer-Verlag, New York, third edition, 2002.

\bibitem[L{\"u}c89]{Lu89}
Wolfgang L{\"u}ck.
\newblock {\em Transformation groups and algebraic {$K$}-theory}, volume 1408
  of {\em Lecture Notes in Mathematics}.
\newblock Springer-Verlag, Berlin, 1989.
\newblock Mathematica Gottingensis.

\bibitem[L{\"u}c05]{Lu05}
Wolfgang L{\"u}ck.
\newblock Survey on classifying spaces for families of subgroups.
\newblock In {\em Infinite groups: geometric, combinatorial and dynamical
  aspects}, volume 248 of {\em Progr. Math.}, pages 269--322. Birkh\"auser,
  Basel, 2005.

\bibitem[LW12]{LW12}
Wolfgang L{\"u}ck and Michael Weiermann.
\newblock On the classifying space of the family of virtually cyclic subgroups.
\newblock {\em Pure Appl. Math. Q.}, 8(2):497--555, 2012.

\bibitem[Mas55]{massey:SPATTFB}
W.~S. Massey.
\newblock {Some problems in algebraic topology and the theory of fibre
  bundles}.
\newblock {\em Ann. of Math. (2)}, 62:327–359, 1955.

\bibitem[McC01]{Mc01}
John McCleary.
\newblock {\em A user's guide to spectral sequences}, volume~58 of {\em
  Cambridge Studies in Advanced Mathematics}.
\newblock Cambridge University Press, Cambridge, second edition, 2001.

\bibitem[ML95]{MacLane:Homotopy}
Saunders Mac~Lane.
\newblock {\em Homology}.
\newblock Classics in Mathematics. Springer-Verlag, Berlin, 1995.
\newblock Reprint of the 1975 edition.

\bibitem[MP02]{MP02}
Conchita Mart{\'\i}nez-P\'erez.
\newblock A spectral sequence in {B}redon (co)homology.
\newblock {\em J. Pure Appl. Algebra}, 176(2-3):161--173, 2002.

\bibitem[MS95]{MS95}
I.~Moerdijk and J.-A. Svensson.
\newblock A {S}hapiro lemma for diagrams of spaces with applications to
  equivariant topology.
\newblock {\em Compositio Math.}, 96(3):249--282, 1995.

\bibitem[MV03]{MV03}
Guido Mislin and Alain Valette.
\newblock {\em Proper group actions and the {B}aum-{C}onnes conjecture}.
\newblock Advanced Courses in Mathematics. CRM Barcelona. Birkh\"auser Verlag,
  Basel, 2003.

\bibitem[Nos17]{Nosaka:QTP}
Takefumi Nosaka.
\newblock {\em Quandles and topological pairs}.
\newblock SpringerBriefs in Mathematics. Springer, Singapore, 2017.
\newblock Symmetry, knots, and cohomology.

\bibitem[SG05]{SG05}
Ruben~Jose Sanchez-Garcia.
\newblock {\em Equivariant {K}-homology of the classifying space for proper
  actions}.
\newblock ProQuest LLC, Ann Arbor, MI, 2005.
\newblock Thesis (Ph.D.)--University of Southampton (United Kingdom).

\bibitem[Sna64]{Snapper:CPRISS}
Ernst Snapper.
\newblock {Cohomology of permutation representations. {I}. {S}pectral
  sequences}.
\newblock {\em J. Math. Mech.}, 13:133–161, 1964.

\bibitem[Tak57]{Takasu-OTCOFITHA}
Satoru Takasu.
\newblock {On the change of rings in the homological algebra}.
\newblock {\em J. Math. Soc. Japan}, 9:315–329, 1957.

\bibitem[Tak59]{takasu-RHARCTOG}
Satoru Takasu.
\newblock {Relative homology and relative cohomology theory of groups}.
\newblock {\em J. Fac. Sci. Univ. Tokyo. Sect. I}, 8:75–110, 1959.

\bibitem[tD87]{tomDieck:TransGrp}
Tammo tom Dieck.
\newblock {\em {Transformation Groups}}.
\newblock Number~8 in {Studies in Mathematics}. Walter de Gruyter, Berlin,
  1987.

\bibitem[Zic09]{Zickert:VCSIR}
Christian~K. Zickert.
\newblock {The volume and {C}hern-{S}imons invariant of a representation}.
\newblock {\em Duke Math. J.}, 150(3):489–532, 2009.

\end{thebibliography}

\end{document}